\documentclass[12pt]{article}
\usepackage[utf8]{inputenc}
\usepackage{geometry}
\usepackage{array} 
\usepackage{paralist} 
\usepackage{verbatim} 
\usepackage{subfig} 

\usepackage{tikz}
\usetikzlibrary{positioning,chains,fit,shapes,calc}
\definecolor{myred}{RGB}{160,80,80}
\definecolor{myblue}{RGB}{80,80,160}

\usepackage{amsmath}
\usepackage{amssymb}
\usepackage{amsthm}
\newtheorem{definition}{Definition}
\newtheorem{conjecture}{Conjecture}
\newtheorem{theorem}{Theorem}
\newtheorem{proposition}{Proposition}
\newtheorem{corollary}{Corollary}

\usepackage{sectsty}
\allsectionsfont{\sffamily\mdseries\upshape} 

\usepackage[nottoc,notlof,notlot]{tocbibind} 
\usepackage[titles,subfigure]{tocloft} 

\usepackage{url}

\title{A conjectured bound on the spanning tree number of bipartite graphs}
\author{
  Michael Slone \\
}

\begin{document}
\maketitle

\begin{abstract}
The Ferrers bound conjecture is a natural graph-theoretic extension of the enumeration of spanning trees for Ferrers graphs.  We document the current status of the conjecture and provide
a further conjecture which implies it.
\end{abstract}

\section{Introduction}

Ferrers graphs are a bipartite analogue of threshold graphs.  
They were introduced by Hammer, Peled, and Srinivasan~\cite{HPS90}, 
where they were called \emph{difference graphs}.
Ferrers graphs and threshold graphs are both realizations of
\emph{Ferrers digraphs}, which were introduced
by Riguet~\cite{Riguet51}.  Threshold graphs and Ferrers graphs obey many analogies.
For example, in the polytope of degree sequences, the 
extreme points correspond to threshold graphs, while in the
polytope of bipartite degree sequences, the extreme points correspond to Ferrers graphs~\cite{PS89b}.

Ehrenborg and van Willigenburg~\cite{Ferrers} studied 
Ferrers graphs and proved that the spanning tree number 
of a Ferrers graph depends only on its degree sequence and
the size of each color class.  This defines a \emph{Ferrers invariant}
for any bipartite graph.  In 2006, Ehrenborg conjectured that
the Ferrers invariant is an upper bound for the spanning tree
number of any bipartite graph.  The purpose of this note is
to document this conjecture and provide some evidence which
suggests the conjecture is reasonable.

First, let us establish some definitions and notation.  The equivalency of the following conditions is demonstrated in~\cite{HPS90}.

\begin{definition}[\cite{HPS90}]
Let $G = (V, E)$ be a bipartite graph with bipartition $V = X \sqcup Y$.  Then $G$ is a \emph{Ferrers graph} if and only if any of the following equivalent conditions hold:
\begin{itemize}
\item
There exist a vertex weighting $w$ and a real number $T$ such that for any $u\ne v\in V$, the vertices $u$ and $v$ are adjacent if and only if $|w(u) - w(v)| \ge T$.

\item
The graph $G'$ constructed from $G$ by adding all possible edges between vertices in $X$ is a threshold graph.

\item
The graph $G$ contains no induced $2K_2$.

\item
The neighborhoods of vertices in $X$ are linearly ordered by inclusion.

\item
The degree sequences for vertices in $X$ and vertices in $Y$ are conjugate.

\end{itemize}
\end{definition}



A sample Ferrers graph is depicted in Figure~\ref{fig:ferrers}.

\begin{figure}[!h]
\caption{The Ferrers graph corresponding to the conjugate degree sequences $(3, 3, 2, 1)$ and $(4, 3, 2)$.}
\label{fig:ferrers}
\centering
\definecolor{myred}{RGB}{160,80,80}
\definecolor{myblue}{RGB}{80,80,160}

\begin{tikzpicture}[thick,
  every node/.style={draw,circle},
  fsnode/.style={fill=black},
  ssnode/.style={fill=black},
  every fit/.style={ellipse,draw,inner sep=-2pt,text width=2cm},
  shorten >= 3pt,shorten <= 3pt
]

\begin{scope}[start chain=going right,node distance=14mm]
\node[fsnode, on chain] (a) {};
\node[fsnode, on chain] (b) {};
\node[fsnode, on chain] (c) {};
\node[fsnode, on chain] (d) {};
\end{scope}

\begin{scope}[yshift=-2cm,start chain=going right,node distance=14mm]
\node[fsnode, on chain] (e) {};
\node[fsnode, on chain] (f) {};
\node[fsnode, on chain] (g) {};
\end{scope}

\draw (a) -- (e);
\draw (a) -- (f);
\draw (a) -- (g);
\draw (b) -- (e);
\draw (b) -- (f);
\draw (b) -- (g);
\draw (c) -- (f);
\draw (c) -- (g);
\draw (d) -- (g);
\end{tikzpicture}
\end{figure}
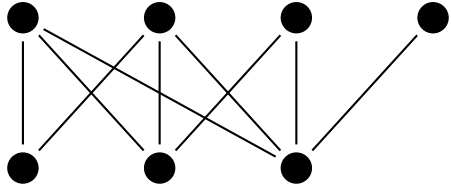

\nocite{Cog82}
\nocite{pak1990enumeration}

Let $T(G)$ denote the spanning tree number of the graph $G$.

\begin{definition}
Let $G = (V, E)$ be a bipartite graph with bipartition $V = X\sqcup Y$.  The \emph{Ferrers invariant} of $G$ is the quantity
\[
	F(G) = \frac{1}{|X| |Y|}\prod_{v\in V}\deg(v).
\]
\end{definition}

\noindent
Ehrenborg and van Willigenburg proved~\cite[Theorem 2.1]{Ferrers} that for Ferrers graphs, $T(G) = F(G)$.

\begin{conjecture}[Ferrers bound conjecture]\label{conj:fbc}
Let $G = (V, E)$ be a bipartite graph with bipartition $V = X\sqcup Y$.  Then
\[
 T(G) \le \frac{1}{|X| |Y|}\prod_{v\in V}\deg(v),
\]
that is, $T(G) \le F(G)$.
\end{conjecture}

\noindent
Let us call a bipartite graph $G$ \emph{Ferrers-good} (respectively \emph{Ferrers-bad}) if $T(G) \le F(G)$ (respectively $T(G) > F(G)$).  Thus Conjecture~\ref{conj:fbc} may be expressed more briefly as the claim that all bipartite graphs are Ferrers-good.

For definitions and notation for majorization-related terms, including reordering of terms,
we follow Marshall, Olkin, and Arnold~\cite{arnold2011inequalities}.

\section{Status of the conjecture}
\nocite{cf2012counting}

Conjecture~\ref{conj:fbc} is trivially true for 
disconnected graphs and all Ferrers graphs.

In 2009, Jack Schmidt (personal communication) computationally verified by an exhaustive search that all bipartite graphs on at most $13$ vertices are Ferrers-good.

In 2013, Praveen Venkataramana proved an inequality weaker than Conjecture~\ref{conj:fbc} valid for all bipartite graphs.  

\begin{proposition}[Venkataramana, unpublished]
Let $G$ be a bipartite graph with red vertices having degrees $d_1,\dots, d_p$ and blue vertices having degrees $e_1,\dots, e_q$.
Then
\[
   T(G) 
\le 
  \prod_{i=2}^p \left(d_i + \frac{1}{2}\right) 
  \prod_{j=2}^q \left(e_i +\frac{1}{2}\right) 
  \sqrt{e_1}
\]
\end{proposition}

In 2014, Garrett and Klee~\cite{gk2014} proved
that Conjecture~\ref{conj:fbc} is equivalent to an inequality on a
particular homogeneous polynomial.  They used this to verify the 
conjecture for trees and all bipartite graphs on at most $11$ 
vertices.

In his 2016 senior thesis, Koo~\cite{k2016}
summarized what was then known about Conjecture~\ref{conj:fbc}.  He proved that
even cycles are Ferrers-good and that the
operation of connecting two graphs by a new
edge preserves Ferrers-goodness.  Moreover, he showed that Conjecture~\ref{conj:fbc} holds for a sufficiently edge-dense graph with a cutvertex of degree 2.

\begin{proposition}[Theorem 5.12 in~\cite{k2016}]
Let $G = (V, E)$ be a bipartite graph with bipartition $V = X \sqcup Y$.  Suppose that $G$ has a cutvertex of degree 2 and furthermore that
\[
\frac{|E|}{|X||Y|} \ge 0.544.
\]
Then $G$ is Ferrers-good.
\end{proposition}

Multiple authors have noted that Ferrers-goodness is preserved under the operation of adding pendant vertices.  The following proposition is somewhat stronger and requires little additional work.

\begin{proposition}\label{prop:cutvertex}
Let $G_1$ and $G_2$ be disjoint Ferrers-good graphs, and let $G$ be the graph formed by gluing $G_1$ and $G_2$ along an arbitrary vertex.  Then $G$ is also Ferrers-good.
\end{proposition}

\begin{proof}
For each $i$, let $G_i = (V_i, E_i)$ have bipartition $V_i = X_i \sqcup Y_i$, and let $x_i\in X_i$ be arbitrary.  Let $G = (V, E)$ be the graph formed from $G_1 \cup G_2$ by identifying $x_1$ with $x_2$, calling the identified vertices $x$.  Thus $G$ has 
bipartition $V = X \sqcup Y$, where $X = X_1 \cup X_2 \cup \{x\} \setminus \{x_1, x_2\}$ and $Y = Y_1 \cup Y_2$.

By construction, $T(G) = T(G_1) \cdot T(G_2)$.
Since each $G_i$ is Ferrers-good, it follows that
\[
T(G)
\le F(G_1)\cdot F(G_2) 
= \frac{\deg_{G_1}(x_1)\deg_{G_2}(x_2)\prod_{v\in V \setminus \{x\}} \deg_G(v)}{|X_1| |X_2| |Y_1| |Y_2|}.
\]
Since $0<\deg_{G_i}(x_i)\le |Y_i|$ and $\deg_G(x) = \deg_{G_1}(x_1) + \deg_{G_2}(x_2)$, it follows that
\[
\frac{\deg_{G_1}(x_1)\deg_{G_2}(x_2)}{\deg_G(x)}
\le
\frac{|Y_1||Y_2|}{|Y|}
\le
\frac{|X_1||X_2|}{|X|}\cdot\frac{|Y_1||Y_2|}{|Y|},
\]
that is,
\[
\frac{\deg_{G_1}(x_1)}{|X_1||Y_1|}\cdot
\frac{\deg_{G_2}(x_2)}{|X_2||Y_2|}
\le
\frac{\deg_G(x)}{|X||Y|}.
\]
Since $\deg_{G_i}(v) = \deg_G(v)$ for any $v\ne x_i$,
\[
F(G_1)\cdot F(G_2)
\le
\frac{1}{|X||Y|}\prod_{v\in V}\deg_G(v)
= F(G).
\]
Hence $G$ is also Ferrers-good.
\end{proof}

\noindent
The following corollary may be useful.

\begin{corollary}\label{cor:cutvertex}
A minimal Ferrers-bad graph is $2$-connected.
\end{corollary}

\begin{proof}
Let $G$ be a minimal Ferrers-bad graph.
If $G$ had a cutvertex $x$, we could decompose it 
into subgraphs $G_1$ and $G_2$ glued along 
the vertex $x$.  By minimality, $G_1$ and $G_2$
are Ferrers-good.  Now apply Proposition~\ref{prop:cutvertex}.
\end{proof}

For a Ferrers-good graph $G$ the following inequality holds:
\begin{eqnarray}\label{eqn:conj2}
 \frac{1}{n}\prod_{i=1}^{n-1}\lambda_i
    &\le
  &\frac{1}{pq}\prod_{i=1}^n d_i,
\end{eqnarray}
where $d$ is the degree sequence of $G$ and
$\lambda$ is the spectrum of the Laplacian.
These possible values of $d$ and $\lambda$ are known to be restricted by 
the Gale--Ryser theorem and the Grone--Merris conjecture, proved by Bai~\cite{bai2011grone}.

\begin{theorem}[Gale~\cite{gale1957theorem}, Ryser~\cite{ryser1957combinatorial}]\label{thm:gale-ryser}
Let $a$ and $b$ be partitions of an integer.  There
is a bipartite graph whose blue degree sequence is $a$
and whose red degree sequence is $b$ if and only if
$a \preceq b^*$.
\end{theorem}

\begin{theorem}[Grone--Merris conjecture, proved in~\cite{bai2011grone}]\label{thm:grone-merris}
The Laplacian spectrum of a graph is majorized by the 
conjugate of its degree sequence.
\end{theorem}

One might hope that the constraints provided by
these theorems would restrict~$\lambda$ enough
for Inequality~\ref{eqn:conj2} to hold.  We can
state this as the following (incorrect) conjecture.

\begin{conjecture}\label{conj:partition}
Let $n = p + q > 1$ be an integer.  Let $a\vdash p$ and $b\vdash q$ be integer partitions, and let $d=d_1\ge d_2\ge\dots\ge d_n$ be their union.
Let $\lambda=\lambda_1\ge\dots\ge\lambda_{n-1}$ be a weakly decreasing sequence of positive real
numbers.

If 
\[
  a \preceq b^*\text{\quad and\quad}d \preceq \lambda \preceq d^*
\]
then 
\[
  \frac{1}{n}\prod_{i=1}^{n-1}\lambda_i
    \le
  \frac{1}{pq}\prod_{i=1}^n d_i
\]
\end{conjecture}

However,
the following simple example, provided by Evan Chen (personal communication), shows that Conjecture~\ref{conj:partition} is false.
Let $d = (2, 2, 2, 1, 1) = (2, 2) \oplus (2, 1, 1)$ (so $a = (2, 2)$ and $b = (2, 1, 1)$) and $\lambda = (2, 2, 2, 2)$.  One can verify that 
the majorization inequalities of Conjecture~\ref{conj:partition} hold.  However,
the conclusion is false:
\[
\frac{1}{5}\prod_{i=1}^4\lambda_i = \frac{16}{5}\not\le \frac{4}{3} = \frac{1}{2\cdot 3}\prod_{i=1}^5 d_i.
\]

\section{Acknowledgments}
The author is grateful to Richard Ehrenborg,
Kathryn Lybarger, Jack Schmidt, and
Duane Skaggs for helpful
conversations, and 
to Evan Chen for providing a simple
counterexample to
a more general conjecture.

\bibliography{mybib}{}
\bibliographystyle{plain}

\end{document}